\documentclass[preprint,dvipdfmx,leqno]{elsarticle}
\usepackage{
amsmath,amssymb,amsthm,ascmac,
bm,natbib,hyperref,here
,txfonts
}
\usepackage[dvipsnames]{xcolor}
\usepackage{tikz}
\usetikzlibrary{intersections,calc}
\newtheoremstyle{mydefinition}
{}{}
{\normalfont}{}
{\bfseries}{}
{\newline}{\thmname{#1}\thmnumber{\ #2}\thmnote{\quad#3}.}
\newtheoremstyle{mydefinition2}
{}{}
{\normalfont}{}
{\bfseries}{}
{\newline}{\thmname{#1}\thmnumber{\ #2}\thmnote{\quad#3}}
\theoremstyle{mydefinition}
\newtheorem{Thm}{Theorem}
\newtheorem{Prop}[Thm]{Proposition}
\newtheorem{Lem}[Thm]{Lemma}
\newtheorem{Def}{Definition}
\theoremstyle{mydefinition2}
\newtheorem{Prop2}[Thm]{Proposition}
\newtheorem{Lem2}[Thm]{Lemma}
\newtheorem*{Lem*}{Lemma}
\makeatletter
\newcommand{\opnorm}{\@ifstar\@opnorms\@opnorm}
\newcommand{\@opnorms}[1]{
\left|\mkern-1.5mu\left|\mkern-1.5mu\left|#1
\right|\mkern-1.5mu\right|\mkern-1.5mu\right|}
\newcommand{\@opnorm}[2][]{
\mathopen{#1|\mkern-1.5mu#1|\mkern-1.5mu#1|}#2
\mathclose{#1|\mkern-1.5mu#1|\mkern-1.5mu#1|}}
{\end{smallmatrix}\right)}
{\end{smallmatrix}\right]}
{\end{smallmatrix}\right\}}
\@addtoreset{equation}{section}
\makeatother
\begin{document}
\newcounter{tmpcnt}
\setcounter{tmpcnt}{1}
\journal{arXiv}
\begin{frontmatter}
\title{Periodic Solutions of the complex Ginzburg-Landau Equation in bounded domains}
\author[add1]{Takanori Kuroda}
\ead{1d\_est\_quod\_est@ruri.waseda.jp}
\author[add2,fn2]{Mitsuharu \^Otani}
\ead{otani@waseda.jp}
\fntext[fn2]{Partly supported by the Grant-in-Aid for Scientific Research, \#18K03382, the Ministry of Education, Culture, Sports, Science and Technology, Japan.}
\address[add1]{Department of Mathematics, School of Science and Engineering, \\ Waseda University, 3-4-1 Okubo Shinjuku-ku, Tokyo, 169-8555, JAPAN}
\address[add2]{Department of Applied Physics, School of Science and Engineering, \\ Waseda University, 3-4-1 Okubo Shinjuku-ku, Tokyo, 169-8555, JAPAN}
\begin{abstract}
In this paper, we are concerned with complex Ginzburg-Landau (CGL) equations.
There are several results on the global existence and smoothing effects of solutions to the initial boundary value problem for (CGL) in bounded or unbounded domains.

In this paper, we study the time periodic problem for (CGL) in bounded domains.
The main strategy in this paper is to regard (CGL) as a parabolic equation with monotone and non-monotone perturbations and to apply non-monotone perturbation theory of parabolic equations developed by \^Otani (1984).
\end{abstract}
\end{frontmatter}
\section{Introduction}\label{sec-1}
In this paper we are concerned with the time periodic problem for following complex Ginzburg-Landau equation in a bounded domain \(\Omega \subset \mathbb{R}^N\) with smooth boundary \(\partial \Omega\):\vspace{-1mm}
\begin{equation}
\tag*{(CGL)}
\left\{
\begin{aligned}
&\partial_t u(t, x) \!-\! (\lambda \!+\! i\alpha)\Delta u \!+\! (\kappa \!+\! i\beta)|u|^{q-2}u \!-\! \gamma u \!=\! f(t, x)
&&\hspace{-2mm}\mbox{in}\ (t, x) \in [0, T] \times \Omega,\\
&u(t, x) = 0&&\hspace{-2mm}\mbox{on}\ (t, x) \in [0, T] \times \partial\Omega,\\
&u(0, x) = u(T,x)&&\hspace{-2mm}\mbox{in}\ x \in \Omega,
\end{aligned}
\right.
\end{equation}
where \(\lambda, \kappa > 0\), \(\alpha, \beta, \gamma \in \mathbb{R}\) are parameters; 
\(i = \sqrt{-1}\) is the imaginary unit; 
\(f: \Omega \times [0, T] \rightarrow \mathbb{C}\) (\(T > 0\)) is a given external force.
Our unknown function \(u:\overline{\Omega} \times [0,\infty) \rightarrow \mathbb{C}\) is complex valued.
Under a suitable assumption on \(\lambda,\kappa,\alpha,\beta\), that is, they belong to the so-called CGL-region (see Figure \ref{CGLR}), we shall establish the existence of periodic solutions without any restriction on \(q\in(2,\infty)\), \(\gamma\in\mathbb{R}^1\) and the size of \(f\).

As extreme cases, (CGL) corresponds to two well-known equations: 
semi-linear heat equations (when \(\alpha=\beta=0\)) and nonlinear Schr\"odinger equations (when \(\lambda=\kappa=0\)).
Thus in general, we can expect (CGL) has specific features of two equations.

Equation (CGL) was introduced by Landau and Ginzburg in 1950 \cite{GL1} as a mathematical model for superconductivity.
Since then it has been revealed that many nonlinear partial differential equations arising from physics can be rewritten in the form of (CGL) (\cite{N1}) as well.

Mathematical studies of the initial boundary value problem for (CGL) are pursued extensively by several authors.
The first treatment was due to Temam \cite{T1}, where weak global solution was constructed by the Galerkin method.
Levermore-Oliver \cite{LO1} constructed weak global solutions to (CGL) on \(N\)-dimensional torus \(\mathbb{T}^N\) by an argument similar to that in the proof of Leray's existence theorem for global weak solutions of Navier-Stokes equations.
They also showed the existence of unique global classical solution of (CGL) for \(u_0 \in {\rm C}^2(\mathbb{T}^N)\) under the conditions \(q \leq 2N/(N-2)_+\) and \(\left(\frac{\alpha}{\lambda}, \frac{\beta}{\kappa}\right) \in {\rm CGL}(c_q^{-1})\) (see Figure \ref{CGLR}).
Ginibre-Velo \cite{GinibreJVeloG1996} showed the existence of global strong solutions for (CGL) in whole space \(\mathbb{R}^N\) under the condition \(\left(\frac{\alpha}{\lambda}, \frac{\beta}{\kappa}\right) \in {\rm CGL}(c_q^{-1})\) without any upper bounds on \(q\) with initial data taken form \(\mathrm{H}^1(\mathbb{R}^N)\cap\mathrm{L}^q(\mathbb{R}^N)\).

Subsequently, Okazawa-Yokota \cite{OY1} treated (CGL) in the framework of the maximal monotone operator theory in complex Hilbert spaces and proved global existence of solutions and a smoothing effect for bounded domains with \(\lambda,\kappa,\alpha,\beta\) belonging to the CGL-region.
Another approach based on the perturbation theory for subdifferential operator was introduced in \cite{KOS1}, where the existence of global solutions together with some smoothing effects in general domains is discussed.

As for the global dynamics of solutions for (CGL) is studied in Takeuchi-Asakawa-Yokota \cite{TAY1}.
They showed the existence of global attractor for bounded domains.

In this paper, we show the existence of time periodic solutions for (CGL) in bounded domains.
Here we regard (CGL) as a parabolic equation governed by the leading term \(-\lambda\Delta u+\kappa|u|^{q-2}u\) which can be described as a subdifferential operator of a certain convex functional, with a monotone perturbation \(-i\alpha\Delta u\) together with a non-monotone perturbation \(i\beta|u|^{q-2}u\).
The time periodic problem for parabolic equations governed by subdfferential operators with non-monotone perturbations is already discussed in \^Otani \cite{O3}, which has a fairly wide applicability.
This theory, however, cannot be applied directly to (CGL) because of the presence of the monotone term, \(-i\alpha\Delta u\).

To cope with this difficulty, we first introduce an auxiliary equation, (CGL) added \(\varepsilon|u|^{r-2}u\) (\(\varepsilon>0,r>q\)) and show the existence of periodic solutions for this equation.
Then letting \(\varepsilon\downarrow0\) with suitable a priori estimates, we obtain desired periodic solutions to (CGL).

This  paper consists of five sections.
In \S2, we fix some notations, prepare some preliminaries and state our main result.
In \S3, we consider an auxiliary equation and show the existence of periodic solutions to the equation.
Our main result is proved in \S4.

\section{Notations and Preliminaries}\label{sec-2}
In this section, we first fix some notations in order to formulate (CGL) as an evolution equation in a real product function space based on the following identification:
\[
\mathbb{C} \ni u_1 + iu_2 \mapsto (u_1, u_2)^{\rm T} \in \mathbb{R}^2.
\]
Moreover we set:
\[
\begin{aligned}
&(U \cdot V)_{\mathbb{R}^2} :=  u_1 v_1 + u_2 v_2,\quad |U|=|U|_{\mathbb{R}^2}, \qquad U=(u_1, u_2)^{\rm T}, \ V=(v_1, v_2)^{\rm T} \in \mathbb{R}^2,\\[1mm]
&\mathbb{L}^2(\Omega) :={\rm L}^2(\Omega) \times {\rm L}^2(\Omega),\quad (U, V)_{\mathbb{L}^2} := (u_1, v_1)_{{\rm L}^2} + (u_2, v_2)_{{\rm L}^2},\\[1mm]
&\qquad U=(u_1, u_2)^{\rm T},\quad V=(v_1, v_2)^{\rm T} \in \mathbb{L}^2(\Omega),\\[1mm]
&\mathbb{L}^r(\Omega) := {\rm L}^r(\Omega) \times {\rm L}^r(\Omega),\quad |U|_{\mathbb{L}^r}^r := |u_1|_{{\rm L}^r}^r + |u_2|_{{\rm L}^r}^r\quad\ U \in \mathbb{L}^r(\Omega)\ (1\leq r < \infty),\\[1mm]
&\mathbb{H}^1_0(\Omega) := {\rm H}^1_0(\Omega) \times {\rm H}^1_0(\Omega),\
(U, V)_{\mathbb{H}^1_0} := (u_1, v_1)_{{\rm H}^1_0} + (u_2, v_2)_{{\rm H}^1_0}\ \ U, V \in \mathbb{H}^1_0(\Omega).
\end{aligned}
\]
We use the differential symbols to indicate differential operators which act on each component of \({\mathbb{H}^1_0}(\Omega)\)-elements:
\[
\begin{aligned}
& D_i = \frac{\partial}{\partial x_i}: \mathbb{H}^1_0(\Omega) \rightarrow \mathbb{L}^2(\Omega),\\
&D_i U = (D_i u_1, D_i u_2)^{\rm T} \in \mathbb{L}^2(\Omega) \ (i=1, \cdots, N),\\[2mm]
& \nabla = \left(\frac{\partial}{\partial x_1}, \cdots, \frac{\partial}{\partial x_N}\right): \mathbb{H}^1_0(\Omega) \rightarrow ({\rm L}^2(\Omega))^{2 N},\\
&\nabla U=(\nabla u_1, \nabla u_2)^T \in ({\rm L}^2(\Omega))^{2 N}.
\end{aligned}
\]

We further define, for \(U=(u_1, u_2)^{\rm T},\ V= (v_1, v_2)^{\rm T},\ W = (w_1, w_2)^{\rm T}\),
\[
\begin{aligned}
&U(x) \cdot \nabla V(x) := u_1(x) \nabla v_1(x) + u_2(x) \nabla v_2(x) \in \mathbb{R}^N,\\[2mm]
&( U(x) \cdot \nabla V(x) ) W(x) := ( u_1(x) \nabla v_1(x)w_1(x) , u_2(x)  \nabla v_2(x) w_2(x))^{\rm T} \in \mathbb{R}^{2N},\\[2mm]
&(\nabla U(x) \cdot \nabla V(x)) := \nabla u_1(x) \cdot \nabla v_1(x) + \nabla u_2(x) \cdot \nabla v_2(x) \in \mathbb{R}^1,\\[2mm]
&|\nabla U(x)| := \left(|\nabla u_1(x)|^2_{\mathbb{R}^N} + |\nabla u_2(x)|^2_{\mathbb{R}^N} \right)^{1/2}.
\end{aligned}
\]

In addition, \(\mathcal{H}^S\) denotes the space of functions with values in \(\mathbb{L}^2(\Omega)\) defined on \([0, S]\) (\(S > 0\)), which is a Hilbert space with the following inner product and norm.
\[
\begin{aligned}
&\mathcal{H}^S := {\rm L}^2(0, S; \mathbb{L}^2(\Omega)) \ni U(t), V(t),\\
&\quad\mbox{with inner product:}\ (U, V)_{\mathcal{H}^S} = \int_0^S (U, V)_{\mathbb{L}^2}^2 dt,\\
&\quad\mbox{and norm:}\ |U|_{\mathcal{H}^S}^2 = (U, U)_{\mathcal{H}^S}.
\end{aligned}
\]

As a realization in \(\mathbb{R}^2\) of the imaginary unit \(i\) in \(\mathbb{C}\), we introduce the following matrix \(I\), which is a linear isometry on \(\mathbb{R}^2\):
\[
I = \begin{pmatrix}
0 & -1\\ 1 & 0
\end{pmatrix}.
\]
We abuse \(I\) for the realization of \(I\) in \(\mathbb{L}^2(\Omega)\), i.e., \(I U = ( - u_2, u_1 )^{\rm T}\) for all \(U = (u_1, u_2)^{\rm T} \in \mathbb{L}^2(\Omega)\).

Then \(I\) satisfies the following properties (see \cite{KOS1}):

\begin{enumerate}
\item Skew-symmetric property:
\begin{equation}
\label{skew-symmetric_property}
(IU \cdot V)_{\mathbb{R}^2} = -\ (U \cdot IV)_{\mathbb{R}^2}; \hspace{4mm}
(IU \cdot U)_{\mathbb{R}^2} = 0 \hspace{4mm}
\mbox{for each}\ U, V \in \mathbb{R}^2.
\end{equation}

\item Commutative property with the differential operator \(D_i = \frac{\partial}{\partial x_i}\):
\begin{equation}
\label{commutative_property}
I D_i = D_i I:\mathbb{H}^1_0 \rightarrow \mathbb{L}^2\ (i=1, \cdots, N).
\end{equation}

\end{enumerate}

Let \({\rm H}\) be a Hilbert space and denote by \(\Phi({\rm H})\) the set of all lower semi-continuous convex function \(\phi\) from \({\rm H}\) into \((-\infty, +\infty]\) such that the effective domain of \(\phi\) given by \({\rm D}(\phi) := \{u \in {\rm H}\mid \ \phi(u) < +\infty \}\) is not empty.
Then for \(\phi \in \Phi({\rm H})\), the subdifferential of \(\phi\) at \(u \in {\rm D}(\phi)\) is defined by
\[
\partial \phi(u) := \{w \in {\rm H}\mid  (w, v - u)_{\rm H} \leq \phi(v)-\phi(u) \hspace{2mm} \mbox{for all}\ v \in {\rm H}\},
\]
which is a possibly multivalued maximal monotone operator with domain\\
\({\rm D}(\partial \phi) = \{u \in {\rm H}\mid  \partial\phi(u) \neq \emptyset\}\).
However for the discussion below, we have only to consider the case where \(\partial \phi\) is single valued.

We define functionals \(\varphi, \ \psi_r:\mathbb{L}^2(\Omega) \rightarrow [0, +\infty]\) (\(r\geq2\)) by
\begin{align}
\label{varphi}
&\varphi(U) :=
\left\{
\begin{aligned}
&\frac{1}{2} \displaystyle\int_\Omega |\nabla U(x)|^2 dx
&&\mbox{if}\ U \in \mathbb{H}^1_0(\Omega),\\[3mm]
&+ \infty &&\mbox{if}\ U \in \mathbb{L}^2(\Omega)\setminus\mathbb{H}^1_0(\Omega),
\end{aligned}
\right.
\\[2mm]
\label{psi}
&\psi_r(U) :=
\left\{
\begin{aligned}
&\frac{1}{r} \displaystyle\int_\Omega |U(x)|_{\mathbb{R}^2}^r dx
&&\mbox{if}\ U \in \mathbb{L}^r(\Omega) \cap \mathbb{L}^2(\Omega),\\[3mm]
&+\infty &&\mbox{if}\ U \in \mathbb{L}^2(\Omega)\setminus\mathbb{L}^r(\Omega).
\end{aligned}
\right.
\end{align}
Then it is easy to see that \(\varphi, \psi_r \in \Phi(\mathbb{L}^2(\Omega))\) and their subdifferentials are given by
\begin{align}
\label{delvaphi}
&\begin{aligned}[t]
&\partial \varphi(U)=-\Delta U\ \mbox{with} \ {\rm D}( \partial \varphi) = \mathbb{H}^1_0(\Omega)\cap\mathbb{H}^2(\Omega),\\[2mm]
\end{aligned}\\
\label{delpsi}
&\partial \psi_r(U) = |U|_{\mathbb{R}^2}^{r-2}U\ {\rm with} \ {\rm D}( \partial \psi_r) = \mathbb{L}^{2(r-1)}(\Omega) \cap \mathbb{L}^2(\Omega).
\end{align}
Furthermore for any $\mu>0$, we can define the Yosida approximations
 \(\partial \varphi_\mu,\ \partial \psi_\mu\) of \(\partial \varphi,\ \partial \psi\) by
\begin{align}
\label{Yosida:varphi}
&\partial \varphi_\mu(U) := \frac{1}{\mu}(U - J_\mu^{\partial \varphi}U) 
= \partial \varphi(J_\mu^{\partial \varphi} U), 
\quad J_\mu^{\partial \varphi} : = ( 1 + \mu\ \partial \varphi)^{-1},
\\[2mm]
\label{Yosida:psi}
&\partial \psi_{r,\mu}(U) := \frac{1}{\mu} (U - J_\mu^{\partial \psi_r} U) 
= \partial \psi_r( J_\mu^{\partial \psi_r} U ), 
\quad  J_\mu^{\partial \psi_r} : = ( 1 + \mu\ \partial \psi_r)^{-1}.
\end{align}
The second identity holds since \(\partial\varphi\) and \(\partial\psi_r\) are single-valued.

Then it is well known that \(\partial \varphi_\mu, \ \partial \psi_{r,\mu}\) are Lipschitz continuous on \(\mathbb{L}^2(\Omega)\) and satisfies the following properties (see \cite{B1}, \cite{B2}):
\begin{align}
\label{asd}
\psi_r(J_\mu^{\partial\psi_r}U)&\leq\psi_{r,\mu}(U) \leq \psi_r(U),\\
\label{as}
|\partial\psi_{r,\mu}(U)|_{{\rm L}^2}&=|\partial\psi_{r}(J_\mu^{\partial\psi}U)|_{{\rm L}^2}\leq|\partial\psi_r(U)|_{\mathbb{L}^2}\quad\forall\ U \in {\rm D}(\partial\psi_r),\ \forall \mu > 0,
\end{align}
where \(\psi_r\) is the Moreau-Yosida regularization of \(\psi\) given by the following formula:
\[
\psi_{r,\mu}(U) = \inf_{V \in \mathbb{L}^2(\Omega)}\left\{
\frac{1}{2\mu}|U-V|_{\mathbb{L}^2}^2+\psi_r(V)
\right\}
=\frac{\mu}{2}|(\partial\psi_r)_\mu(U)|_{\mathbb{L}^2}^2+\psi_r(J_\mu^{\partial\psi}U)\geq0.
\]
Moreover since \(\psi_r(0)=0\), it follows from the definition of subdifferential operators and \eqref{asd} that
\begin{align}\label{asdf}
(\partial\psi_{r,\mu}(U),U)_{\mathbb{L}^2} = (\partial\psi_{r,\mu}(U),U-0)_{\mathbb{L}^2}
\leq \psi_{r,\mu}(U) - \psi_{r,\mu}(0) \leq \psi_r(U).
\end{align}

Here for later use, we prepare some fundamental properties of \(I\) in connection with \(\partial \varphi,\ \partial \psi_r,\  \partial \varphi_\mu,\ \partial \psi_{r,\mu}\).
\begin{Lem2}[(c.f. \cite{KOS1} Lemma 2.1)]
The following angle conditions hold.
\label{Lem:2.1}
\begin{align}
\label{orth:IU}
&(\partial \varphi(U), I U)_{\mathbb{L}^2} = 0\quad 
\forall U \in {\rm D}(\partial \varphi),\quad 
(\partial \psi_r(U), I U)_{\mathbb{L}^2} = 0\quad \forall U \in {\rm D}(\partial \psi_r), 
\\[2mm]
\label{orth:mu:IU}
&(\partial \varphi_\mu(U), I U)_{\mathbb{L}^2} = 0,\quad 
(\partial \psi_{r,\mu}(U), I U)_{\mathbb{L}^2} = 0 \quad 
\forall U \in \mathbb{L}^2(\Omega), 
\\[2mm]
\label{orth:Ipsi}
&\begin{aligned}
&(\partial \psi_q(U), I \partial \psi_r(U))_{\mathbb{L}^2} = 0,\\
&(\partial \psi_q(U), I \partial \psi_{r,\mu}(U))_{\mathbb{L}^2}=0\quad 
\forall U \in {\rm D}(\partial \psi_r), \forall q,r \geq 2,
\end{aligned}\\
\label{angle}
&(\partial\varphi(U),\partial\psi_r(U))_{\mathbb{L}^2} \geq 0.
\end{align}
\end{Lem2}

\begin{proof}
We only give a proof of the second property in \eqref{orth:Ipsi} here.
Let \(W=J_\mu^{\partial\psi_r}U\), then \(U = W + \mu\partial\psi_r(W)\).
It holds
\[
(\partial \psi_q(U), I \partial \psi_{r,\mu}(U))_{\mathbb{L}^2}
=
(|U|^{q-2}(W + \mu|W|^{r-2}W),I|W|^{r-2}W)_{\mathbb{L}^2}=0.
\]
\end{proof}

Thus (CGL) can be reduced to the following evolution equation:
\[
\tag*{(ACGL)}
\left\{
\begin{aligned}
&\frac{dU}{dt}(t) \!+\! \lambda\partial\varphi(U) \!+\! \alpha I \partial \varphi(U) \!+\! (\kappa+ \beta I) \partial \psi_q(U) \!-\! \gamma U \!=\! F(t),\quad t \in (0,T),\\
&U(0) =U(T),
\end{aligned}
\right.
\]
where \(f(t, x) = f_1(t, x) + i f_2(t, x)\) is identified with \(F(t) = (f_1(t, \cdot), f_2(t, \cdot))^{\rm T} \in \mathbb{L}^2(\Omega)\).

In order to state our main results, we introduce the following region:
\[
\begin{aligned}
\label{CGLR}
{\rm CGL}(r) &:= \left\{(x,y) \in \mathbb{R}^2 \mid xy \geq 0\ \mbox{or}\ \frac{|xy| - 1}{|x| + |y|} < r\right\}\\
&= {\rm S}_1(r) \cup {\rm S}_2(r) \cup\ {\rm S}_3(r) \cup {\rm S}_4(r),
\end{aligned}
\]
where \({\rm S}_i(\cdot)\) (\(i=1,2,3,4\)) are given by
\begin{equation}
\label{CGLRS1-4}
\begin{aligned}
&{\rm S}_1(r) := \left\{(x, y) \in \mathbb{R}^2 \mid |x| \leq r\right\},\\
&{\rm S}_2(r) := \left\{(x, y) \in \mathbb{R}^2 \mid |y| \leq r\right\},\\
&{\rm S}_3(r) := \left\{(x, y) \in \mathbb{R}^2 \mid xy > 0\right\},\\
&{\rm S}_4(r) := \left\{(x, y) \in \mathbb{R}^2 \mid |1 + xy| < r |x - y|\right\}.
\end{aligned}
\end{equation}
This region is frequently referred as the CGL region.
Also, we use the parameter \(c_q \in [0,\infty)\) measuring the strength of 
the nonlinearity:
\[
\label{cq}
c_q := \frac{q - 2}{2\sqrt{q - 1}}.
\]

This exponent \(c_q\) will play an important role in what follows, which is based on the following inequality.
\begin{Lem2}[(c.f. \cite{KOS1} Lemma 4.1)] 
    The following inequalities hold for all \label{key_inequality}
    $U \in {\rm D}(\partial \varphi) \cap
    {\rm D}(\partial \psi_q)$:
    \begin{align}
     \label{key_inequality_1}
     &|(\partial \varphi(U), 
     I \partial \psi_q(U))_{\mathbb{L}^2}|
     \leq c_{q}(\partial \varphi(U), 
     \partial \psi_q(U))_{\mathbb{L}^2}.
    \end{align}
   \end{Lem2}

In this paper we are concerned with periodic solutions of (ACGL) in the following sense:
\begin{Def}[Periodic solutions]
A function \(U \in {\rm C}([0,T];\mathbb{L}^2(\Omega))\) is a periodic solution of (ACGL) if the following conditions are satisfied:
\begin{enumerate}\renewcommand{\labelenumi}{(\roman{enumi})}
\item \(U \in {\rm D}(\partial\varphi)\cap{\rm D}(\partial\psi_q)\) and satisfies (ACGL) for a.e. \(t\in(0,T)\),
\item \(\frac{dU}{dt},\partial\varphi(U),\partial\psi_q(U)\in{\rm L}^2(0,T;\mathbb{L}^2(\Omega))\),
\item \(\varphi(U(t))\) and \(\psi_q(U(t))\) are absolutely continuous on \([0,T]\),
\item \(U(0)=U(T)\).
\end{enumerate}
\end{Def}
We note that condition (iii) follows from (ii) and hence periodic solutions \(U\) belong to \({\rm C}([0,T];\mathbb{L}^q(\Omega)\cap\mathbb{H}^1_0(\Omega))\).

Our main results can be stated as follows.

\begin{Thm}[Existence of Periodic Solutions]\label{MTHM}
Let \(\Omega \subset \mathbb{R}^N\) be a bounded domain of \({\rm C}^2\)-regular class.
Let \((\frac{\alpha}{\lambda}, \frac{\beta}{\kappa}) \in {\rm CGL}(c_q^{-1})\) and \(\gamma \in \mathbb{R}\).
Then for all \(F \in {\rm L}^2(0, T;\mathbb{L}^2(\Omega))\) with given \(T > 0\), there exists a periodic solution to (ACGL).
\end{Thm}

\section{Auxiliary Problems}\label{sec-4}
In this section, we consider the following auxiliary equation:
\[
\tag*{(AE)\(_\varepsilon\)}
\left\{
\begin{aligned}
&\frac{dU}{dt}(t) \!+\! \lambda\partial\varphi(U)\!+\!\alpha I\partial\varphi(U)\!+\!\varepsilon\partial\psi_r(U)\!+\!(\kappa\!+\! \beta I) \partial \psi_q(U) \!-\! \gamma U \!=\! F(t),\ t \!\in\! (0,T),\\
&U(0) = U(T),
\end{aligned}
\right.
\]
with \(r > q\) and \(\varepsilon>0\).
Then we have
\begin{Prop}
Let \(F \in \mathcal{H}^T\), \(\varepsilon>0\) and \(r > q > 2\).\label{GWP}
Then there exists a periodic solution for (AE)\(_\varepsilon\).
\end{Prop}

In order to prove this proposition, for a given \(h\in\mathcal{H}^T\) we consider Cauchy problem of the form:
\[
\tag*{(IVP)\(_\mu^h\)}
\left\{
\begin{aligned}
&\frac{dU}{dt}(t) \!+\! \lambda\partial\varphi(U)\!+\!\varepsilon\partial\psi_r(U)\!+\!\kappa\partial \psi_q(U)\!+\!\alpha I\partial\varphi_\mu(U)\!+\! U \!-\!h(t) \!=\! F(t),\ t \!\in\! (0,T),\\
&U(0) = U_0,
\end{aligned}
\right.
\]
where \(U_0 \in \mathbb{L}^2(\Omega)\).
We claim that this equation has a unique solution \(U = U^h \in {\rm C}([0,T];\mathbb{L}^2(\Omega))\) satisfying the following regularities
\begin{enumerate}\renewcommand{\labelenumi}{(\roman{enumi})}
\item $U \in {\rm W}^{1,2}_{\rm loc}((0,T);\mathbb{L}^2(\Omega))$, 
\item $U(t) \in {\rm D}(\partial \varphi) \cap {\rm D}(\partial \psi_q) \cap {\rm D}(\partial \psi_r)$ for a.e. $t \in (0,T)$
	   and satisfies (AE)\(_\mu\) for a.e. $t \in (0,T)$,
\item $\varphi(U(\cdot))$, \(\psi_q(U(\cdot))\), \(\psi_r(U(\cdot)) \in
	   {\rm L}^1(0,T)\) and $t\varphi(U(t))$, \(t\psi_q(U(t))\), \(t\psi_r(U(t))
	   \in {\rm L}^\infty(0,T)\),
\item $\sqrt{t} \frac{d}{dt}U(t)$,
	   $\sqrt{t} \partial \varphi(U(t))$,
	   $\sqrt{t} \partial \psi_q(U(t))$, 
	   $\sqrt{t} \partial \psi_r(U(t)) 
	   \in {\rm L}^2(0,T;\mathbb{L}^2(\Omega))$.
\end{enumerate}

Since \(\partial\varphi_\mu(U)\) is a Lipschitz perturbation, to ensure the above claim, we only have to check the following:
\begin{Lem}\label{asta}
The operator \(\lambda\partial\varphi + \varepsilon\partial\psi_r + \kappa\partial\psi_q\) is maximal monotone in \(\mathbb{L}^2(\Omega)\) and satisfies
\begin{equation}
\label{ast}
\lambda\partial\varphi + \varepsilon\partial\psi_r + \kappa\partial\psi_q
=
\partial(\lambda\varphi+\varepsilon\psi_r+\kappa\psi_q).
\end{equation}
\end{Lem}
Since \(\lambda\partial\varphi+\varepsilon\partial\psi_r+\kappa\partial\psi_q\) is monotone, and \(\partial(\lambda\varphi+\varepsilon\psi_r+\kappa\psi_q)\subset\lambda\partial\varphi+\varepsilon\partial\psi_r+\kappa\partial\psi_q\), to show \eqref{ast}, it suffices to check that \(\lambda\partial\varphi+\varepsilon\partial\psi_r+\kappa\partial\psi_q\) is maximal monotone.

For this purpose, we rely on the following Proposition as in the proof of Lemma 2.3. in \cite{KOS1}.
\begin{Prop2}[(Br\'ezis, H. \cite{B2} Theorem 9)]
\label{angler}
Let \(B\) be maximal monotone in \({\rm H}\) and \(\phi \in \Phi({\rm H})\).
Suppose
\begin{align}
\phi((1+\mu B)^{-1}u) \leq \phi(u), \hspace{4mm} \forall \mu>0 \hspace{2mm} \forall u \in {\rm D}(\phi).
\label{angle1}
\end{align}
Then \(\partial \phi + B\) is maximal monotone in  \({\rm H}\).
\end{Prop2}
\begin{proof}[Proof of Lemma \ref{asta}]
We note that \(\varepsilon\partial\psi_q+\kappa\partial \psi_q\) is obviously maximal monotone in \(\mathbb{L}^2(\Omega)\).
First we show \((1 + \mu \{\varepsilon\partial\psi_q+\kappa\partial \psi_q\})^{-1}{\rm D}(\varphi) \subset {\rm D}(\varphi)\), where \({\rm D}(\varphi) = \mathbb{H}^1_0(\Omega)\).
Let \(U \in \mathbb{C}^1_\mathrm{c}(\Omega) := {\rm C}^1_0(\Omega) \times {\rm C}^1_0(\Omega)\) and \(V:=(1+\mu\{\varepsilon\partial\psi_q+\kappa\partial \psi_q\})^{-1}U\), which implies \(V(x)+\mu \{\varepsilon|V(x)|^{r-2}_{\mathbb{R}^2}V(x)+\kappa|V(x)|_{\mathbb{R}^2}^{q-2}V(x)\}=U(x)\) for a.e. \(x \in \Omega\).
Here define \(G:\mathbb{R}^2 \rightarrow \mathbb{R}^2\) by \(G : V \mapsto G(V) =V+\mu \{\varepsilon|V|^{r-2}_{\mathbb{R}^2}V+\kappa|V|_{\mathbb{R}^2}^{q-2}V\}\), then we get \(G(V(x))=U(x)\).
Note that G is of class \({\rm C}^1\) and bijective from \(\mathbb{R}^2\) into itself and its Jacobian determinant is given by
\[
\begin{aligned}
&\det D G(V)\\
&= (1 + \mu \{\varepsilon|V|_{\mathbb{R}^2}^{r-2}+\kappa|V|^{q-2}\})
(1 + \mu\{\varepsilon(r-1)|V|^{r-2}+\kappa (q-1) |V|_{\mathbb{R}^2}^{q-2}\}) \neq 0\\
&\qquad \mbox{for each}\ V \in \mathbb{R}^2.
\end{aligned}
\]
Applying the inverse function theorem, we have \(G^{-1} \in {\rm C}^1(\mathbb{R}^2;\mathbb{R}^2)\). Hence \(V(\cdot)=G^{-1}(U(\cdot)) \in \mathbb{C}^1_\mathrm{c}(\Omega),\) which implies \((1 + \mu\{\varepsilon\partial\psi_q+\kappa\partial \psi_q\})^{-1} \mathbb{C}^1_\mathrm{c}(\Omega) \subset \mathbb{C}^1_\mathrm{c}(\Omega)\).
Now let \(U_n \in \mathbb{C}^1_\mathrm{c}(\Omega)\) and \(U_n \rightarrow U\) in \(\mathbb{H}^1(\Omega)\).
Then \(V_n := (1 + \mu \{\varepsilon\partial\psi_q+\kappa\partial \psi_q\})^{-1}U_n \in \mathbb{C}^1_\mathrm{c}(\Omega)\) satisfy
\[
\begin{aligned}
|V_n - V|_{\mathbb{L}^2}
&= |(1 + \mu \{\varepsilon\partial\psi_q+\kappa\partial \psi_q\})^{-1}U_n - (1+\mu \{\varepsilon\partial\psi_q+\kappa\partial \psi_q\})^{-1}U|_{\mathbb{L}^2}\\ 
&\leq |U_n - U|_{\mathbb{L}^2} \rightarrow 0 \hspace{4mm} \mbox{as}\ n \rightarrow \infty,
\end{aligned}
\]
whence it follows that \(V_n \rightarrow V\) in \(\mathbb{L}^2(\Omega)\).
Also differentiation of \(G(V_n(x))=U_n(x)\) gives
\begin{equation}
\label{adfjakfjdkaj}
\begin{aligned}
&(1+\mu\{\varepsilon|V_n(x)|^{r-2}+\kappa|V_n(x)|_{\mathbb{R}^2}^{q-2}\})\nabla V_n(x)\\
&+ \mu\{\varepsilon(r-2)|V_n(x)|_{\mathbb{R}^2}^{r-4}+\kappa(q-2)|V_n(x)|_{\mathbb{R}^2}^{q-4}\} (V_n(x) \cdot \nabla V_n(x)) ~\! V_n(x) = \nabla U_n(x).
\end{aligned}
\end{equation}
Multiplying \eqref{adfjakfjdkaj} by \(\nabla V_n(x)\), we easily get \(|\nabla V_n(x)|^2 \leq (\nabla U_n(x) \cdot \nabla V_n(x))\).
Therefore by the Cauchy-Schwarz inequality, we have \(\varphi(V_n) \leq \varphi(U_n) \rightarrow \varphi(U)\).
Thus the boundedness of \(\{ |\nabla V_n| \}\) in \({\rm L}^2\) assures that \(V_n \to V\) weakly in \(\mathbb{H}^1_0(\Omega)\), hence we have \((1+\mu \partial \psi)^{-1}{\rm D}(\varphi) \subset {\rm D}(\varphi)\).
Furthermore, from the lower semi-continuity of the norm in the weak topology, we derive \(\varphi(V) \leq \varphi(U)\).
This is nothing but the desired inequality \eqref{angle1}.
\end{proof}

Thus by the standard argument of subdifferential operator theory (see Br\'ezis \cite{B1,B2}), we have a unique solution \(U=U_\mu^h\) of (IVP)\(_\mu^h\) satisfying (i)-(iv).
Then by letting \(\mu\downarrow0\), we can easily show that \(U_\mu^h\) converges to the unique solution\(U^h\) (satisfying regularity (i)-(iv)) of the following Cauchy problem (cf. proof of Theorem 2 of \cite{KOS1}):

So far we have the unique solution of the following initial value problem:
\[
\tag*{(IVP)\(^h\)}
\left\{
\begin{aligned}
&\frac{dU}{dt}(t) \!+\! \lambda\partial\varphi(U)\!+\!\varepsilon\partial\psi_r(U)\!+\!\kappa\partial \psi_q(U)\!+\!\alpha I\partial\varphi(U)\!+\!U\!-\!h(t) \!=\! F(t),\ t \!\in\! (0,T),\\
&U(0) = U_0.
\end{aligned}
\right.
\]

For all \(U_0, V_0 \in \mathbb{L}^2(\Omega)\) and the corresponding solutions \(U(t), V(t)\) of (IVP)\(^h\), by the monotonicity of \(\partial\varphi, I\partial\varphi, \partial\psi_r, \partial\psi_q\), we easily obtain
\[
\frac{1}{2}\frac{d}{dt}|U(t)-V(t)|_{\mathbb{L}^2}^2+|U(t)-V(t)|_{\mathbb{L}^2}^2\leq 0,
\]
whence follows
\[
|U(T) - V(T)|_{\mathbb{L}^2} \leq e^{-T}|U_0 - V_0|_{\mathbb{L}^2}.
\]
Then the Poincar\'e map: \(\mathbb{L}^2(\Omega) \ni U_0 \mapsto U(T) \in \mathbb{L}^2(\Omega)\) becomes a strict contraction.
Therefore the fixed point of the Poincar\'e map gives the unique periodic solution of (AE)\(^h_\varepsilon\).
\[
\tag*{(AE)\(^h_\varepsilon\)}
\left\{
\begin{aligned}
&\frac{dU}{dt}(t) \!+\! \lambda\partial\varphi(U)\!+\!\varepsilon\partial\psi_r(U)\!+\!\kappa\partial \psi_q(U)\!+\!\alpha I\partial\varphi(U)\!+\! U \!-\!h(t) \!=\! F(t),\ t \!\in\! (0,T),\\
&U(0) = U(T).
\end{aligned}
\right.
\]
Here we note that since \(U(T) \in {\rm D}(\varphi)\cap{\rm D}(\psi_q)\cap{\rm D}(\psi_r)\) by (iii), we automatically have \(U(0)\in {\rm D}(\varphi)\cap{\rm D}(\psi_q)\cap{\rm D}(\psi_r)\).

We next define the mapping
\[
\mathcal{F}: \mathcal{H}^T \supset {\rm B}_R \ni h \mapsto U_h \mapsto \beta I\partial\psi_q(U) - (\gamma + 1)U \in {\rm B}_R \subset \mathcal{H}^T,
\]
where \({\rm B}_R\) is the ball in \(\mathcal{H}^T\) centered at the origin with radius \(R > 0\), to be fixed later, and \(U_h\) is the unique solution of (AE)\(_\varepsilon^h\) with given \(h \in {\rm B}_R\).

In order to ensure \(\mathcal{F}(h) \in {\rm B}_R\), we are going to establish a priori estimates for solutions \(U=U_h\) of (AE)\(_\varepsilon^h\).
\begin{Lem}
Let \(U = U_h\) be the periodic solution for (AE)\(_\varepsilon^h\).
Then there exist a constant \(C\) constant depending only on \(|\Omega|\), \(T\), \(r\), \(\varepsilon\) and \(|F|_{\mathcal{H}^T}\) such that
\begin{equation}\label{plm}
\sup_{t\in[0,T]}|U(t)|_{\mathbb{L}^2} \leq C + C|h|_{\mathcal{H}^T}^{\frac{1}{r-1}} + C|h|_{\mathcal{H}^T}^{\frac{r}{2(r-1)}}.
\end{equation}
\end{Lem}
\begin{proof}
Multiplying (AE)\(_\varepsilon^h\) by the solution \(U\), we obtain
\[
\begin{aligned}
\frac{1}{2}\frac{d}{dt}|U|_{\mathbb{L}^2}^2 + \varepsilon|\Omega|^{1 - \frac{r}{2}}|U|_{\mathbb{L}^2}^r
&\leq \frac{1}{2}\frac{d}{dt}|U|_{\mathbb{L}^2}^2 + 2\lambda\varphi(U) + r\varepsilon\psi_r(U) + q\kappa\psi_q(U) + |U|_{\mathbb{L}^2}^2\\
&\leq (|F|_{\mathbb{L}^2} + |h|_{\mathbb{L}^2})|U|_{\mathbb{L}^2},
\end{aligned}
\]
where we used \(|\Omega|^{1 - \frac{r}{2}}|U|_{\mathbb{L}^2}^r \leq r\psi_r(U)\).
Moreover by Young's inequality, we have
\begin{equation}
\label{wer}
\begin{aligned}
\frac{1}{2}\frac{d}{dt}|U|_{\mathbb{L}^2}^2 + \frac{\varepsilon|\Omega|^{1 - \frac{r}{2}}}{2}|U|_{\mathbb{L}^2}^r
\leq 
C_1
(|F|_{\mathbb{L}^2}^{\frac{r}{r-1}} + |h|_{\mathbb{L}^2}^{\frac{r}{r-1}}),&\\
C_1=\left(1-\frac{1}{r}\right)\left(\frac{r\varepsilon|\Omega|^{1-\frac{r}{2}}}{4}\right)^{-\frac{1}{r-1}}.&
\end{aligned}
\end{equation}

Put \(m = \min_{0 \leq t \leq T}|U(t)|_{\mathbb{L}^2}\) and \(M = \max_{0 \leq t \leq T}|U(t)|_{\mathbb{L}^2}\).
Then we have
\[
\begin{aligned}
M^2 &\leq m^2 + 2C_1
\int_0^T(|F|_{\mathbb{L}^2}^{\frac{r}{r-1}} + |h|_{\mathbb{L}^2}^{\frac{r}{r-1}})dt\\
&\leq m^2 + 2C_1
T^{\frac{r-2}{2(r-1)}}
(|F|_{\mathcal{H}^T}^{\frac{r}{r-1}}  + |h|_{\mathcal{H}^T}^{\frac{r}{r-1}}),
\end{aligned}
\]
whence follows
\begin{equation}
\label{asda}
M \leq m + \tilde{C}_1(|F|_{\mathcal{H}^T}^{\frac{r}{2(r-1)}}  + |h|_{\mathcal{H}^T}^{\frac{r}{2(r-1)}}),
\end{equation}
where
\[
\tilde{C}_1=
\sqrt{2}C_1^{\frac{1}{2}}T^{\frac{r-2}{r-1}}.
\]

On the other hand, integrating \eqref{wer} over \((0,T)\), we obtain
\[
\frac{1}{2}T\varepsilon|\Omega|^{1 - \frac{r}{2}}m^r \leq 
C_1T^{\frac{r-2}{2(r-1)}}
(|F|_{\mathcal{H}^T}^{\frac{r}{r-1}}  + |h|_{\mathcal{H}^T}^{\frac{r}{r-1}})
\]
which implies
\begin{equation}
\label{lkj}
m \leq C_2(|F|_{\mathcal{H}^T}^{\frac{1}{r-1}}  + |h|_{\mathcal{H}^T}^{\frac{1}{r-1}}),
\end{equation}
where
\[
C_2=
\left(\frac{2C_1}{\varepsilon|\Omega|^{1 - \frac{r}{2}}}\right)^{\frac{1}{r}}\frac{1}{T^{\frac{1}{2(r-1)}}}.
\]

Combining \eqref{asda} with \eqref{lkj}, we have the desired inequality.
\end{proof}

We note that \(r > 2\) implies \(\frac{1}{r-1}<1\) and \(\frac{r}{2(r-1)} < 1\).

\begin{Lem}
Let \(U = U_h\) be the periodic solution for (AE)\(^h_\varepsilon\).
Then there exists a constant \(C\) depending only on \(|\Omega|\), \(T\), \(r\), \(\lambda\), \(\alpha\), \(\varepsilon\) and \(|F|_{\mathcal{H}^T}\) such that
\begin{equation}\label{zaqw}
\sup_{t\in[0,T]}\varphi(U(t))
+\int_0^T|\partial\varphi(U(t))|_{\mathbb{L}^2}^2dt
+\int_0^T|\partial\psi_r(U(t))|_{\mathbb{L}^2}^2dt
\leq C + C|h|_{\mathcal{H}^T}^{2}.
\end{equation}
\end{Lem}
\begin{proof}\mbox{}\hspace{\parindent}
Multiplying (AE)\(_\varepsilon^h\) by \(\partial\varphi(U)\), then by using \eqref{skew-symmetric_property} and \eqref{angle}, we get
\begin{equation}\label{zz}
\frac{d}{dt}\varphi(U) + \frac{\lambda}{2}|\partial\varphi(U)|_{\mathbb{L}^2}^2 + 2\varphi(U)\leq \frac{1}{\lambda}(|F|_{\mathbb{L}^2}^2 + |h|_{\mathbb{L}^2}^2).
\end{equation}

Set \(m_1 = \min_{0 \leq t \leq T}\varphi(U)\) and \(M_1 = \max_{0 \leq t \leq T}\varphi(U)\).
Then we have
\[
M_1 \leq m_1 + \frac{1}{\lambda}(|F|_{\mathcal{H}^T}^2 + |h|_{\mathcal{H}^T}^2).
\]
On the other hand, integrating \eqref{zz} over \((0,T)\), we have
\begin{equation}\label{zaq}
\frac{\lambda}{2}\int_0^T|\partial\varphi(U)|_{\mathbb{L}^2}^2dt+
2Tm_1 \leq \frac{1}{\lambda}(|F|_{\mathcal{H}^T}^2 + |h|_{\mathcal{H}^T}^2),
\end{equation}
whence follows
\begin{equation}\label{za}
M_1=\max_{t\in[0,T]}\varphi(U(t)) \leq \left(1 + \frac{1}{2T}\right)\frac{1}{\lambda}(|F|_{\mathcal{H}^T}^2 + |h|_{\mathcal{H}^T}^2).
\end{equation}

Next we multiply (AE)\(_\varepsilon^h\) by \(\partial\psi_r(U)\), then in view of \((\partial\psi_r(U),\partial\psi_q(U))_{\mathbb{L}^2}\geq0\) and \eqref{angle}, we get
\begin{equation}\label{zzz}
\frac{d}{dt}\partial\psi_r(U) + \frac{\varepsilon}{4}|\partial\psi_r(U)|_{\mathbb{L}^2}^2 \leq \frac{1}{\varepsilon}(\alpha^2|\partial\varphi(U)|_{\mathbb{L}^2}^2 + |F|_{\mathbb{L}^2}^2 + |h|_{\mathbb{L}^2}^2).
\end{equation}
Integrating \eqref{zzz} with respect to \(t\) over \((0,T)\), we obtain by \eqref{zaq}
\[
\frac{\varepsilon}{4}
|\partial\psi_r(U)|_{\mathcal{H}^T}^2=
\frac{\varepsilon}{4}\int_0^T|\partial\psi_r(U)|^2_{\mathbb{L}^2}dt
\leq
\frac{1}{\varepsilon}\left(1 + \frac{2\alpha^2}{\lambda^2}\right)
(|F|_{\mathcal{H}^2}^2 + |h|_{\mathcal{H}^2}^2).
\]
\end{proof}

\begin{proof}[Proof of Proposition \ref{GWP}]
By the interpolation inequality, we find that for any \(\eta>0\), there exists \(C_\eta>0\) such that
\begin{equation}\label{zaqws}
|\partial\psi_q(U)|_{\mathbb{L}^2}^2 \leq |\partial\psi_r(U)|_{\mathbb{L}^2}^{2\frac{q-2}{r-2}}|U|_{\mathbb{L}^2}^{2\frac{r-q}{r-2}}
\leq \eta|\partial\psi_r(U)|_{\mathbb{L}^2}^2 + C_\eta|U|_{\mathbb{L}^2}^2.
\end{equation}

Hence, by virtue of \eqref{orth:IU}, \eqref{plm} and \eqref{zaqw}, we get
\[
|\mathcal{F}(h)|_{\mathcal{H}^T}^2
\begin{aligned}[t]
&= |\beta I \partial\psi_q(U) - (\gamma + 1) U|_{\mathcal{H}^T}^2\\
&= |\beta|^2|\partial\psi_q(U)|_{\mathcal{H}^T}^2 + |\gamma + 1|^2|U|_{\mathcal{H}^T}^2\\
&\leq \eta|\beta|^2\{C + C|h|_{\mathcal{H}^T}^{2}\}
+ (|\gamma + 1|^2+C_\eta|\beta|^2) T\left\{ C + C|h|_{\mathcal{H}^T}^{\frac{1}{r-1}} + C|h|_{\mathcal{H}^T}^{\frac{r}{2(r-1)}}\right\}^2.
\end{aligned}
\]
Here we fix \(\eta\) such that
\[
\eta = \frac{1}{2}|\beta|^{-2}C^{-1}
\]
and take a sufficient large \(R\) such that
\[
\eta|\beta|^2C + \frac{1}{2}R^2 + (|\gamma + 1|^2+C_\eta|\beta|^2) T\left\{ C + CR^{\frac{1}{r-1}} + CR^{\frac{r}{2(r-1)}}\right\}^2 \leq R^2.
\]
Thus we conclude that \(\mathcal{F}\) maps \(\mathrm{B}_R\) into itself.

Next we ensure that \(\mathcal{F}\) is continuous with respect to the weak topology of \(\mathcal{H}^T\).
Let \(h_n \rightharpoonup h\) weakly in \({\rm B}_R \subset {\rm L}^2(0,T;\mathbb{L}^2(\Omega))\) and let \(U_n\) be the unique periodic solution of (AE)\(_\varepsilon^{h_n}\).

The estimates \eqref{plm}, \eqref{zaqw} and Rellich-Kondrachov's theorem ensure that \(\{U_n(t)\}_{n\in\mathbb{N}}\) is precompact in \(\mathbb{L}^2(\Omega)\) for all \(t\in[0,T]\).
On the other hand, from estimates \eqref{zaqw}, \eqref{zaqws} and equation (AE)\(_\varepsilon^h\), we derive
\begin{equation}\label{zaqwsx}
\int_0^T\left|\frac{dU_n}{dt}(t)\right|^2dt\leq C
\end{equation}
for a suitable constant \(C\), whence it follows that \(\{U_n(t)\}_{n\in\mathbb{N}}\) forms an equi-continuous family in \(\mathrm{C}([0,T];\mathbb{L}^2(\Omega))\).
Hence 
we can apply Ascoli's theorem 
to obtain a strong convergent subsequence in \({\rm C}([0,T];\mathbb{L}^2(\Omega))\) (denoted again by \(\{U_n\}\)).

Thus by virtue of \eqref{zaqw} and the demi-closedness of operators \(\partial\varphi,\partial\psi_q,\partial\psi_r,\frac{d}{dt}\), we obtain
\begin{align}
U_n&\to U&&\mbox{strongly in}\ {\rm C}([0,T];\mathbb{L}^2(\Omega)),\\
\partial\varphi(U_n)&\rightharpoonup \partial\varphi(U)&&\mbox{weakly in}\ {\rm L}^2(0,T;\mathbb{L}^2(\Omega)),\\
\partial\psi_q(U_n)&\rightharpoonup \partial\psi_q(U)&&\mbox{weakly in}\ {\rm L}^2(0,T;\mathbb{L}^2(\Omega)),\\
\label{pl}\partial\psi_r(U_n)&\rightharpoonup \partial\psi_r(U)&&\mbox{weakly in}\ {\rm L}^2(0,T;\mathbb{L}^2(\Omega)),\\
\frac{dU_n}{dt}&\rightharpoonup \frac{dU}{dt}&&\mbox{weakly in}\ {\rm L}^2(0,T;\mathbb{L}^2(\Omega)).
\end{align}
Consequently \(U\) satisfies
\[
\tag*{(AE)\(_\varepsilon^h\)}
\left\{
\begin{aligned}
&\frac{dU}{dt}(t) \!+\! \lambda\partial\varphi(U)\!+\!\varepsilon\partial\psi_r(U)\!+\!\kappa\partial \psi_q(U)\!+\!\alpha I\partial\varphi(U)\!+\! U \!-\!h(t) \!=\! F(t),\ t \!\in\! (0,T),\\
&U(0) = U(T),
\end{aligned}
\right.
\]
that is, \(U\) is the unique periodic solution of (AE)\(_h\).
Since the argument above does not depend on the choice of subsequences, we can conclude that \(\mathcal{F}\) is weakly continuous in \(\mathcal{H}^T={\rm L}^2(0, T;\mathbb{L}^2(\Omega))\).

Therefore by Schauder's fixed point theorem, we obtain a fixed point of the mapping \(\mathcal{F}\) which give the desired periodic solution of (AE)\(_\varepsilon\).
\end{proof}

\section{Proof of Theorem \ref{MTHM}}
In this section, we discuss the convergence of periodic solutions \(U_\varepsilon\) of (AE)\(_\varepsilon\) as \(\varepsilon\to0\), by establishing a priori estimates independent of \(\varepsilon\).
\[
\tag*{(AE)\(_\varepsilon\)}
\left\{
\begin{aligned}
&\frac{dU}{dt}(t) \!+\! \lambda\partial\varphi(U)\!+\!\alpha I\partial\varphi(U)\!+\!\varepsilon\partial\psi_r(U)\!+\!(\kappa\!+\! \beta I) \partial \psi_q(U) \!-\! \gamma U \!=\! F(t),\ t \!\in\! (0,T),\\
&U(0) = U(T),
\end{aligned}
\right.
\]
with \(r > q\) and \(\varepsilon>0\).

To this end, we mainly rely on our key inequality \eqref{key_inequality_1} and repeat much the same arguments as those in our previous paper \cite{KOS1}.

\begin{Lem}
Let \(U = U_\varepsilon\) be the periodic solution for (AE)\(_\varepsilon\). \label{1st_energy}
Then there exists a constant \(C\) depending only on \(|\Omega|\), \(T\), \(\lambda,\kappa,\gamma\) \(q\) and \(|F|_{\mathcal{H}^T}\) but not on \(\varepsilon\) such that
\begin{equation}\label{1e}
\sup_{t\in[0,T]}|U(t)|_{\mathbb{L}^2}^2
+\int_0^T\varphi(U(t))dt + \varepsilon\int_0^T\psi_r(U(t))dt + \int_0^T\psi_q(U(t))dt \leq C.
\end{equation}
\end{Lem}
\begin{proof}
Multiplying (AE)\(_\varepsilon\) by \(U\), we get by \eqref{orth:IU}
\[
\begin{aligned}
&\frac{1}{2}\frac{d}{dt}|U|_{\mathbb{L}^2}^2 + \frac{\kappa}{2}|\Omega|^{1 - \frac{q}{2}}|U|_{\mathbb{L}^2}^q
+2\lambda\varphi(U) + r\varepsilon\psi_r(U) + \frac{q\kappa}{2}\psi_q(U)-\gamma|U|_{\mathbb{L}^2}^2\\
&\leq \frac{1}{2}\frac{d}{dt}|U|_{\mathbb{L}^2}^2 + 2\lambda\varphi(U) + r\varepsilon\psi_r(U) + q\kappa\psi_q(U)-\gamma|U|_{\mathbb{L}^2}^2\\
&\leq |F|_{\mathbb{L}^2}|U|_{\mathbb{L}^2},
\end{aligned}
\]
where we used \(|\Omega|^{1 - \frac{q}{2}}|U|_{\mathbb{L}^2}^q\leq q\psi_q(U)\).
Since there exists a constant \(C_3\) such that
\[
\gamma|U|_{\mathbb{L}^2}^2\leq\frac{\kappa}{4}|\Omega|^{1-\frac{q}{2}}|U|_{\mathbb{L}^2}^q+C_3,
\]
we obtain
\[
\begin{aligned}
&\frac{1}{2}\frac{d}{dt}|U|_{\mathbb{L}^2}^2 + \frac{\kappa|\Omega|^{1 - \frac{q}{2}}}{4}|U|_{\mathbb{L}^2}^q
+2\lambda\varphi(U) + r\varepsilon\psi_r(U) + \frac{q\kappa}{2}\psi_q(U)\\
&\leq |F|_{\mathbb{L}^2}|U|_{\mathbb{L}^2}+C_3.
\end{aligned}
\]
Moreover by Young's inequality, we have
\begin{equation}
\label{wers}
\begin{aligned}
&\frac{1}{2}\frac{d}{dt}|U|_{\mathbb{L}^2}^2 + \frac{\kappa|\Omega|^{1 - \frac{q}{2}}}{8}|U|_{\mathbb{L}^2}^q
+2\lambda\varphi(U) + r\varepsilon\psi_r(U) + \frac{q\kappa}{2}\psi_q(U)\\
&\leq
\left(1-\frac{1}{q}\right)
\left(\frac{q\kappa|\Omega|^{1-\frac{q}{2}}}{8}\right)^{-\frac{1}{q-1}}
|F|_{\mathbb{L}^2}^{\frac{q}{q-1}}+C_3.
\end{aligned}
\end{equation}

Put \(m = \min_{0 \leq t \leq T}|U(t)|_{\mathbb{L}^2}\) and \(M = \max_{0 \leq t \leq T}|U(t)|_{\mathbb{L}^2}\).
Then we have
\[
\begin{aligned}
M^2 &\leq m^2 + 
2\left(1-\frac{1}{q}\right)
\left(\frac{q\kappa|\Omega|^{1-\frac{q}{2}}}{8}\right)^{-\frac{1}{q-1}}
\int_0^T|F|_{\mathbb{L}^2}^{\frac{q}{q-1}}dt + 2C_3T\\
&\leq m^2 + 
2\left(1-\frac{1}{q}\right)
\left(\frac{q\kappa|\Omega|^{1-\frac{q}{2}}}{8}\right)^{-\frac{1}{q-1}}
T^{\frac{q-2}{2(q-1)}}
|F|_{\mathcal{H}^T}^{\frac{q}{q-1}} + 2C_3T,
\end{aligned}
\]
whence follows
\begin{equation}
\label{asdas}
M \leq m +
\sqrt{2}\left(1-\frac{1}{q}\right)^{\frac{1}{2}}
\left(\frac{q\kappa|\Omega|^{1-\frac{q}{2}}}{8}\right)^{-\frac{1}{2(q-1)}}
T^{\frac{q-2}{4(q-1)}}
|F|_{\mathcal{H}^T}^{\frac{q}{2(q-1)}}
+\sqrt{2C_3T}.
\end{equation}

Consequently, integrating \eqref{wers} with respect to \(t\) over \((0,T)\), we obtain
\[
\begin{aligned}
&\frac{\kappa|\Omega|^{1 - \frac{q}{2}}}{8}Tm^q
+2\lambda\int_0^T\varphi(U(t))dt + r\varepsilon\int_0^T\psi_r(U(t))dt + \frac{q\kappa}{2}\int_0^T\psi_q(U(t))dt\\
&\leq
\left(\frac{q\kappa|\Omega|^{1-\frac{q}{2}}}{8}\right)^{-\frac{1}{q-1}}
T^{\frac{q-2}{2(q-1)}}
|F|_{\mathcal{H}^T}^{\frac{q}{q-1}}
+C_3T,
\end{aligned}
\]
which implies
\begin{equation}
\label{lkjs}
m \leq
\left[
\frac{8}{T\kappa|\Omega|^{1 - \frac{q}{2}}}
\left\{
\left(\frac{q\kappa|\Omega|^{1-\frac{q}{2}}}{8}\right)^{-\frac{1}{q-1}}
T^{\frac{q-2}{2(q-1)}}
|F|_{\mathcal{H}^T}^{\frac{q}{q-1}}
+C_3T
\right\}
\right]^{\frac{1}{q}}.
\end{equation}

Thus \eqref{1e} follows from \eqref{asdas} and \eqref{lkjs}.
\end{proof}

\begin{Lem}\label{2nd_energy}
Let \(U = U_\varepsilon\) be the periodic solution of (AE)\(_\varepsilon\) and assume \((\frac{\alpha}{\lambda},\frac{\beta}{\kappa})\in{\rm CGL}(c_q^{-1})\).
Then there exists a constant \(C\) depending only on \(|\Omega|\), \(T\), \(q\), \(\lambda,\kappa,\alpha,\beta,\gamma\) and \(|F|_{\mathcal{H}^T}\) but not on \(\varepsilon\) such that
\begin{equation}\label{zaqwss}
\begin{aligned}
&\sup_{t \in [0,T]}\varphi(U(t))
+\sup_{t \in [0, T]}\psi_q(U(t))
+\sup_{t \in [0, T]}\varepsilon \psi_r(U(t))\\
&+\int_0^T|\partial \varphi(U(t))|_{\mathbb{L}^2}^2dt
+\int_0^T|\partial \psi_q(U(t))|_{\mathbb{L}^2}^2dt\\
&+\int_0^T \left| \frac{dU(t)}{dt}\right|_{\mathbb{L}^2}^2 dt
+ \varepsilon^2\int_0^T|\partial\psi_r(U(t))|_{\mathbb{L}^2}^2dt
\leq C.
\end{aligned}
\end{equation}
\end{Lem}
\begin{proof}
Multiplication of (AE)\(_\varepsilon\) by 
   $\partial \varphi(U)$ and   
     $\partial \psi_q(U)$ together with \eqref{angle} and \eqref{skew-symmetric_property} 
         give 
   \begin{align} \label{afdidfhafdkjftt} 
    & 
\frac{d}{dt} \varphi(U)
       + \lambda |\partial \varphi(U)|_{\mathbb{L}^2}^2
         + \kappa G
          + \beta B
            \leq 2 \gamma_+ \varphi(U)
              + (F, \partial \varphi(U))_{\mathbb{L}^2},
\\
    \label{fsdksdfk}
        &
\frac{d}{dt} \psi_q(U(t))
      + \kappa |\partial \psi_q(U)|_{\mathbb{L}^2}^2
        + \lambda G
          - \alpha  B
            \leq q \gamma_+ \psi_q(U(t))
              + (F, \partial \psi_q(U))_{\mathbb{L}^2},
   \end{align}
   where
   $\gamma_+:=\max \{\gamma, 0\}$ and
   \begin{equation*}
     G:=(\partial \varphi(U), 
          \partial \psi_q(U))_{\mathbb{L}^2}, \quad 
         B:=(\partial \varphi(U), 
               I \partial \psi_q(U))_{\mathbb{L}^2}.
   \end{equation*}
 We add \eqref{afdidfhafdkjftt}$\times \delta^2$ to \eqref{fsdksdfk}
   for some $\delta >0$ to get
\begin{equation}\label{sfdsldflsdfsl}
   \begin{aligned}
    &\frac{d}{dt} \left\{
      \delta^2 \varphi(U) + \psi_q(U) \right\}
        + \delta^2 \lambda |\partial \varphi(U)|_{\mathbb{L}^2}^2
          + \kappa |\partial \psi_q(U)|_{\mathbb{L}^2}^2
\\
&+(\delta^2 \kappa+\lambda) G
        +(\delta^2 \beta - \alpha )B
\\
&       \leq  \gamma_+\left\{
         2 \delta^2 \varphi(U) + q \psi_q(U)  \right\}
            + (F, \delta^2 \partial \varphi(U)
                + \partial \psi_q(U))_{\mathbb{L}^2}.
   \end{aligned}
\end{equation}
Here we introduce an another parameter \(\epsilon\in(0,\min\{\lambda, \kappa\})\).
   By the inequality of arithmetic and geometric means,
   and the Bessel' inequality, 
   we have
\begin{equation}
   \begin{aligned}
    &\delta^2 \lambda |\partial \varphi(U)|_{\mathbb{L}^2}^2
    + \kappa |\partial \psi_q(U)|_{\mathbb{L}^2}^2
\\
&    =\epsilon 
    \left\{
    \delta^2  |\partial \varphi(U)|_{\mathbb{L}^2}^2
    + |\partial \psi_q(U)|_{\mathbb{L}^2}^2
    \right\}
    +(\lambda- \epsilon) \delta^2 |\partial 
    \varphi(U)|_{\mathbb{L}^2}^2
    +(\kappa  - \epsilon) |\partial
    \psi_q(U)|_{\mathbb{L}^2}^2
\\
&    \geq \epsilon 
    \left\{
    \delta^2  |\partial \varphi(U)|_{\mathbb{L}^2}^2
    + |\partial \psi_q(U)|_{\mathbb{L}^2}^2
    \right\}
    +2\sqrt{
    (\lambda-\epsilon)(\kappa-\epsilon)\delta^2 |\partial 
    \varphi(U)|_{\mathbb{L}^2}^2 |\partial 
    \psi_q(U)|_{\mathbb{L}^2}^2
    }
\\
    \label{fdfkskgkd}
 &   \geq 
    \epsilon
    \left\{
    \delta^2  |\partial \varphi(U)|_{\mathbb{L}^2}^2
    + |\partial \psi_q(U)|_{\mathbb{L}^2}^2
    \right\}
    +2\sqrt{
    (\lambda-\epsilon)(\kappa-\epsilon)\delta^2
    (G^2+B^2)
    }.
   \end{aligned}
\end{equation}
   We here recall the key inequality
   \eqref{key_inequality_1}
   \begin{align}
    \label{fadfks}
    G \geq c_q^{-1}|B|.
   \end{align}
 Hence  
   \eqref{sfdsldflsdfsl}, \eqref{fdfkskgkd}
   and \eqref{fadfks} yield
\begin{equation}
   \begin{aligned}
&\frac{d}{dt} \left\{
    \delta^2 \varphi(U) \!+\! \psi_q(U) \right\}
    \!+\!\epsilon
    \left\{
    \delta^2  |\partial \varphi(U)|_{\mathbb{L}^2}^2
    \!+\! |\partial \psi_q(U)|_{\mathbb{L}^2}^2
    \right\}
    \!+\!J(\delta, \epsilon)|B|
\\
    \label{sfdsldflsksfdhsl}
&    \leq 
\gamma_+\left\{
    2 \delta^2 \varphi(U) \!+\!q\psi_q(U)  \right\}
    \!+\! (F, \delta^2 \partial \varphi(U)
    \!+\! \partial \psi_q(U))_{\mathbb{L}^2}.
   \end{aligned}
\end{equation}
   where
   \begin{align*}
    J(\delta, \epsilon):=2 \delta \sqrt{
    (1+c_q^{-2})(\lambda-\epsilon)(\kappa-\epsilon)}
    + c_q^{-1}(\delta^2 \kappa+\lambda)
    -|\delta^2 \beta - \alpha|.
   \end{align*}

 Now we are going to show that 
   $(\frac{\alpha}{\lambda}, \frac{\beta}{\kappa})
   \in {\rm CGL}(c_q^{-1})$ assures 
   $J(\delta, \epsilon) \geq 0$
   for some $\delta$ and $\epsilon$.
   By the continuity of 
   $ J(\delta, \cdot) : \epsilon \mapsto J(\delta, \epsilon)$
   it suffices to show
   $J(\delta, 0) > 0$
   for some $\delta$.
   When $\alpha \beta >0$,
   it is enough to take
   $\delta=\sqrt{\alpha / \beta}$.
   When $\alpha \beta \leq 0$,
   we have
   $|\delta^2 \beta - \alpha|
   =\delta^2 |\beta|+ |\alpha|$.
  Hence
   \begin{align*}
    J(\delta, 0) 
    = (c_q^{-1}\kappa -|\beta|)\delta^2
    +2 \delta \sqrt{(1+c_q^{-2})\lambda \kappa} 
    +(c_q^{-1}\lambda-|\alpha|).
   \end{align*}
 Therefore if $|\beta|/ \kappa \leq c_q^{-1}$,
   we get $J(\delta, 0) > 0$
   for sufficiently large $\delta >0$.
   If $c_q^{-1} < |\beta| / \kappa$,
   we find that it is enough to see the 
   discriminant is positive:
   \begin{align}
    D/4:=(1+c_q^{-2})\lambda \kappa
    -(c_q^{-1}\kappa -|\beta|)
    (c_q^{-1}\lambda-|\alpha|)>0.
   \end{align}
 Since
   \begin{align*}
    D/4>0 
    \Leftrightarrow
    \frac{|\alpha|}{\lambda}\frac{|\beta|}{\kappa}-1
    <c_q^{-1}\left(
    \frac{|\alpha|}{\lambda}+\frac{|\beta|}{\kappa}
    \right),
   \end{align*}
   the condition
   $(\frac{\alpha}{\lambda}, \frac{\beta}{\kappa})
   \in {\rm CGL}(c_q^{-1})$ yields
   $D>0$,
   whence $J(\delta, 0)>0$ for some $\delta\neq0$.

Now we take $\delta\neq0$ and $\epsilon>0$ such that 
   $J(\delta, \epsilon) \geq 0$.

Applying Young's inequality to \eqref{sfdsldflsksfdhsl}, we obtain
\begin{equation}
    \label{qazs}
   \begin{aligned}
&\frac{d}{dt} \left\{
    \delta^2 \varphi(U) \!+\! \psi_q(U) \right\}
    \!+\!\frac{\epsilon}{2}
    \left\{
    \delta^2  |\partial \varphi(U)|_{\mathbb{L}^2}^2
    \!+\! |\partial \psi_q(U)|_{\mathbb{L}^2}^2
    \right\}
    \!+\!J(\delta, \epsilon)|B|
\\
&    \leq 
\gamma_+\left\{
    2 \delta^2 \varphi(U) \!+\!q\psi_q(U)  \right\}
    \!+\! \frac{1+\delta^2}{2\epsilon}|F|_{\mathbb{L}^2}^2.
   \end{aligned}
\end{equation}
We integrate \eqref{qazs} with respect to \(t\) over \((0,T)\) and then by Lemma \ref{1st_energy} we obtain
\begin{equation}\label{plmj}
\int_0^T|\partial \varphi(U(t))|_{\mathbb{L}^2}^2dt
+\int_0^T|\partial \psi_q(U(t))|_{\mathbb{L}^2}^2dt\leq C.
\end{equation}

Multiplying (AE)\(_\varepsilon\) by \(\varepsilon\partial\psi_r(U)\) and applying Young's inequality, we obtain by \eqref{orth:Ipsi} and \eqref{angle},
\begin{equation}\label{plmk}
\varepsilon\frac{d}{dt}\psi_r(U) + \frac{\varepsilon^2}{2}|\partial\psi_r(U)|_{\mathbb{L}^2}^2 \leq r\gamma_+\varepsilon\psi_r(U) + \alpha^2|\partial\varphi(U)|_{\mathbb{L}^2}^2 + |F(t)|_{\mathbb{L}^2}^2.
\end{equation}
We integrate \eqref{plmk} with respect to \(t\) over \((0,T)\), then by \eqref{1e} and \eqref{plmj}, we obtain
\begin{equation}
\varepsilon^2\int_0^T|\partial \psi_r(U(t))|_{\mathbb{L}^2}^2dt\leq C.
\end{equation}
Hence by equation (AE)\(_\varepsilon\), we can estimate the time derivative of solutions, i.e.,
\begin{equation}
\int_0^T\left|\frac{dU(t)}{dt}\right|_{\mathbb{L}^2}^2dt\leq C.
\end{equation}

Now we are going to derive a priori estimates for the first three terms in \eqref{zaqwss}.
Applying \eqref{angle} and Young's inequality to \eqref{afdidfhafdkjftt} and \eqref{fsdksdfk}, we obtain
\begin{align} \label{wsx} 
    & 
\begin{aligned}
&\frac{d}{dt} \varphi(U)
       + \frac{\lambda}{2} |\partial \varphi(U)|_{\mathbb{L}^2}^2 +\varphi(U)\\
&            \leq (2 \gamma_+ +1)\varphi(U)
              + \frac{1}{\lambda}|F|_{\mathbb{L}^2}^2 + \frac{\beta^2}{\lambda}|\partial\psi_q(U)|_{\mathbb{L}^2}^2,
\end{aligned}\\
    \label{edc}
        &
\begin{aligned}
&\frac{d}{dt} \psi_q(U(t))
      + \frac{\kappa}{2} |\partial \psi_q(U)|_{\mathbb{L}^2}^2 + \psi_q(U(t))\\
&             \leq (q \gamma_++1) \psi_q(U(t))
              + \frac{1}{\kappa}|F|_{\mathbb{L}^2}^2+ \frac{\alpha^2}{\kappa}|\partial\varphi(U)|_{\mathbb{L}^2}^2.
\end{aligned}
   \end{align}
Moreover we modify \eqref{plmk} as
\begin{equation}
\begin{aligned}
&\varepsilon\frac{d}{dt}\psi_r(U) + \frac{\varepsilon^2}{2}|\partial\psi_r(U)|_{\mathbb{L}^2}^2+\varepsilon\psi_r(U)\\ 
&\leq (r\gamma_++1)\varepsilon\psi_r(U) + \alpha^2|\partial\varphi(U)|_{\mathbb{L}^2}^2 + |F(t)|_{\mathbb{L}^2}^2.
\end{aligned}
\end{equation}

Since the arguments for deducing estimates concerning \(\sup_{t\in[0,T]}\psi_q(U(t))\) and \(\sup_{t\in[0,T]}\varepsilon\psi_r(U(t))\) are the same as that for \(\sup_{t\in[0,T]}\varphi(U(t))\), we here only show how to deduce the estimate for \(\sup_{t\in[0,T]}\varphi(U(t))\).

Set \(m_1 = \min_{0 \leq t \leq T}\varphi(U) = \varphi(U(t_1))\) and \(M_1 = \max_{0 \leq t \leq T}\varphi(U)=\varphi(U(t_2))\) with \(t_2\in(t_1,t_1+T]\).
Then we integrate \eqref{wsx} with respect to \(t\) on \((t_1,t_2)\).
Noting \eqref{1e} and \eqref{plmj}, we obtain
\[
M_1 \leq m_1 + C.
\]
On the other hand, integrating \eqref{wsx} with respect to \(t\) over \((0,T)\), by \eqref{1e} and \eqref{plmj} we obtain
\[
m_1T \leq C,
\]
whence the following estimate holds:
\[
M_1 \leq \left(1+\frac{1}{T}\right)C,
\]
which gives the desired estimate.
\end{proof}
\begin{proof}[Proof of Theorem \ref{MTHM}]
By \eqref{1e}, \eqref{zaqwss} and Rellich-Kondrachov's theorem, \(\{U_\varepsilon(t)\}_{\varepsilon>0}\) forms a compact set for all \(t\in[0,T]\).
Moreover the \(\mathrm{L}^2(0,T;\mathbb{L}^2(\Omega))\) estimate for \(\frac{dU}{dt}\) in \eqref{zaqwss} ensures that \(\{U_\varepsilon\}_{\varepsilon>0}\) is equi-continuous.
Hence by Ascoli's theorem, there exists subsequence \(\{U_n\}_{n\in\mathbb{N}}:=\{U_{\varepsilon_n}\}_{n\in\mathbb{N}}\) of \(\{U_\varepsilon(t)\}_{\varepsilon>0}\) which converges strongly in \({\rm C}([0,T];\mathbb{L}^2(\Omega))\).

On the other hand, by \eqref{zaqwss} and the demi-closedness of \(\partial\varphi,\partial\psi_q,\frac{d}{dt}\) we can extract a subsequence of \(\{U_n\}_{n\in\mathbb{N}}\) denoted again by \(\{U_n\}_{n\in\mathbb{N}}\) such that
\begin{align}\label{rfv}
U_n&\to U&&\mbox{strongly in}\ {\rm C}([0,T];\mathbb{L}^2(\Omega)),\\
\partial\varphi(U_n)&\rightharpoonup \partial\varphi(U)&&\mbox{weakly in}\ {\rm L}^2(0,T;\mathbb{L}^2(\Omega)),\\
\partial\psi_q(U_n)&\rightharpoonup \partial\psi_q(U)&&\mbox{weakly in}\ {\rm L}^2(0,T;\mathbb{L}^2(\Omega)),\\
\varepsilon_n\partial\psi_r(U_n)&\rightharpoonup g&&\mbox{weakly in}\ {\rm L}^2(0,T;\mathbb{L}^2(\Omega)),\\
\frac{dU_n}{dt}&\rightharpoonup \frac{dU}{dt}&&\mbox{weakly in}\ {\rm L}^2(0,T;\mathbb{L}^2(\Omega)),
\end{align}
where \(g \in {\rm L}^2(0,T;\mathbb{L}^2(\Omega))\).

What we have to do is to show that \(g=0\).
Due to \eqref{zaqwss}, we get
\[
\begin{aligned}
|\varepsilon|U_\varepsilon|^{r - 2}U_\varepsilon|_{\mathbb{L}^{\frac{r}{r - 1}}}^{\frac{r}{r - 1}}= \int_\Omega\varepsilon^{\frac{r}{r - 1}}|U_\varepsilon|^rdx= \varepsilon^{\frac{1}{r - 1}}\varepsilon\int_\Omega|U_\varepsilon|^rdx \leq \varepsilon^{\frac{1}{r - 1}}rC \to 0&\\
\mbox{as}\ \varepsilon \to 0\quad\mbox{uniformly on}\ [0,T],&
\end{aligned}
\]
which yields
\begin{equation}
\label{edelpsir-Uerr-1}
\varepsilon_n|U_n|^{r - 2}U_n \rightarrow 0\quad\mbox{strongly in}\ \mathbb{L}^{\frac{r}{r - 1}}(\Omega)\quad\mbox{uniformly on}\ [0, T].
\end{equation}
Hence
\[
\varepsilon_n|U_n|^{r - 2}U_n \rightarrow 0=g\quad\mbox{in}\ \mathcal{D}'((0,T)\times\Omega).
\]

Therefore \(U\) satisfies the equation (ACGL) and the convergence \eqref{rfv} ensures that \(U(0)=U(T)\).
\end{proof}


\end{document}